\documentclass{jaac}
\def\currentvolume{*}
\def\currentissue{*}
\def\currentyear{*}
\def\currentmonth{*}
\def\doi{Not assigned}
\firstpage{1}
\usepackage{booktabs}
\usepackage{float}
\usepackage{multirow}
\usepackage{enumerate}

\def\shorttitle{An inertial shrinking projection algorithm}

\def\shortauthor{Z. Zhou,  B. Tan \& S. Li}

\title{\uppercase{An inertial shrinking projection algorithm for split common fixed point problems}}

\author{ Zheng Zhou$^{1,\dag}$, Bing Tan$^1$ and Songxiao Li$^{1}$}

\date{}

\begin{document}
\baselineskip 12pt

\maketitle
\begin{abstract}
In this paper, the purpose is to introduce and study a new modified shrinking projection algorithm with inertial effects, which solves split common fixed point problems in Banach spaces. The corresponding strong convergence theorems are obtained without the assumption of semi-compactness on mappings. Finally, some numerical examples are presented to illustrate the results in this paper.
\end{abstract}


\begin{MSC}
47H09, 47H10, 47J25, 47N40.
\end{MSC}

\thispagestyle{first}\renewcommand{\thefootnote}{\fnsymbol{footnote}}
\footnotetext{\hspace*{-5mm}
\renewcommand{\arraystretch}{1}
\begin{tabular}{@{}r@{}p{10cm}@{}}
$^\dag$& the corresponding author. Email address: zhouzheng2272@163.com (Z. Zhou)\\
$^1$&Institute of Fundamental and Frontier Sciences, University of Electronic Science and Technology of China, Chengdu, China\\
Note:& This paper has been accepted by Journal of Applied Analysis and Computation.
\end{tabular}}

\vspace{-2mm}

\section{Introduction}
To model inverse problems in phase retrievals and medical image reconstruction \cite{s1}, Censor and Elfving \cite{s2} introduced the concept of the split feasibility problem (for short, SFP) in framework of finite-dimensional Hilbert spaces in 1994. It has been founded that the SFP can be used in many areas of applications, such as image restoration, computer tomograph, radiation therapy treatment planning and other areas of mathematical research \cite{s3,s4,s5,s6,s7,s8,lab1}.

As a generalization of the SFP in 2009, Censor and Segal \cite{s9} introduced the following split common fixed point problem (for short, SCFPP): Let $H_1$ and $H_2$ be two Hilbert spaces and $A:H_1\rightarrow H_2$ be a bounded linear operator, $T:H_1\rightarrow H_1$ and $S:H_2\rightarrow H_2$ be two mappings ($F(T)$ and $F(S)$ denote the fixed point sets of $T$ and $S$, respectively). The split common fixed point problem for mappings $T$ and $S$ which is to find a point $x^*$ satisfing
\begin{equation}
\label{eqq}
	x^*\in F(T) \text{ and } Ax^{*}\in F(S).
\end{equation}
The solution set of (1.1) is denoted by $\Gamma$, i.e., $\Gamma=\{x^*| x^*\in F(T),Ax^*\in F(S)\}$.

Since then, the SCFPP has been widely studied by many authors in Hilbert spaces (see \cite{lab5,lab6,s12,lab2,lab3}). Usually in order to achieve strong convergence properties of the SCFPP, we often do this directly by considering the assumption of semi-compactness on the mappings. In addition, there are papers using some algorithms to replace the assumption of the mappings, such as Halpern iterative algorithm, viscosity iterative algorithm, shrinking projection algorithm. Ulteriorly, there are a few studies on the split common fixed point problem in the framework of Banach spaces. For example in recent years, Takahashi and Yao \cite{s22} and Takahashi \cite{s23,s24} studied the split common fixed point problem in the setting of one Hilbert space and one Banach space, and obtained some weak convergence theorems and strong convergence theorems. Furthermore, in the framework of two Banach spaces, there are many studies established as follows:

In 2015, Tang et al. \cite{s25} studied the SCFPP for asymptotically nonexpansive mappings and quasi-strictly pseudo-contractive mappings. Then, the strong convergence theorem was also proved with the condition of semi-compactness on the mappings. In 2016, Shehu et al. \cite{s26} studied split feasibility problems and fixed point problems for left Bregman strongly nonexpansive mappings, and showed strong convergence theorems by Halpern iterative method.

Recently, Ma et al. \cite{s28} also studied split feasibility problems and fixed point problems in Banach spaces, and obtained the strong convergence theorem by the following shrinking projection iterative algorithm

\[
\left\{
\begin{aligned}
&{z_{n} =J_1^{-1}(J_1x_n+\gamma A^*J_2(P_Q-I)Ax_n),}\\
&{y_{n}=J_1^{-1}((1-\alpha_n)J_1z_n+\alpha_nJ_1Sz_n),}\\
&{C_{n+1}=\{v\in C_n:\phi(v,y_n)\leq\phi(v,x_n), \phi(v,z_n)\leq\phi(v,x_n)\},}\\
&{x_{n+1}=\Pi_{C_{n+1}}x_1, n\geq 1,}
\end{aligned}
\right.
\]
where $E_1$ is a $ 2 $-uniformly convex and $ 2 $-uniformly smooth real Banach space, $E_2$ is a smooth, strictly convex and reflective Banach space and $Q$ is a nonempty closed convex subset of $E_2$, $A:E_1\rightarrow E_2$ is a bounded linear operator with the adjoint operator $A^*$, $S:E_1\rightarrow E_1$ is a closed quasi-$\phi$-nonexpansive mapping, $P_Q:E_2\rightarrow Q$ is the metric projection and $\Pi_{C_{n+1}}:E_1\rightarrow C_{n+1}$ is the generalized projection.

In view of the above studies and methods and in order to accelerate better the convergence rate of the iterative algorithms. The inertial effects have been studied recently by many authors in terms of variational inequality problems, inclusion problems, equilibrium problems, etc., see \cite{ceng2019inertial,shehu2019modified,s29,lab4,s32,s38} and the references therein. The main characteristic of the inertial method is that the new iterate process is produced by making use of two values of the previous iterative point. In 2001, Alvarez and Attouch \cite{s29} studied the problem of approximating the null point of a maximal monotone operator and proposed the following inertial proximal algorithm:

\[
\left\{
\begin{aligned}
&{y_n=x_n +\alpha_n (x_n-x_{n-1}),}\\
&{x_{n+1}=(I+\lambda T)^{-1}y_n, \forall n \geq 1.}
\end{aligned}
\right.
\]
Then, they obtained the weak convergence of the algorithm.

For these research, the ideas of this paper are as follows: this article introduce a new shrinking projection iterative algorithm with inertial effects to solve   problem \eqref{eqq} for firmly nonexpansive-like mappings in the framework of $p$-uniformly convex and uniformly smooth real Banach spaces. Meanwhile, the strong convergence theorems of this algorithm are obtained without assumption of semi-compactness on mappings. As applications, the results are utilized to split fixed point problems and variational inclusion problems, split fixed point problem and equilibrium problems. Furthermore, some numerical examples are used to demonstrate and show the efficiency of our main results. To this end, some basic properties and relevant lemmas will be introduced in next section which will be used in the proof for the convergence analysis of the proposed algorithm.

\section{Preliminaries}

Throughout this paper, we use notations: $\rightarrow$ to denote the strong convergence and $\rightharpoonup$ to denote the weak convergence. The set of all fixed points of $T$ is denoted by $F(T)$.

Let $E$ be a Banach space. A function $\delta_E:[0,2]\rightarrow[0,1]$  is called the modulus of convexity of $E$  as follows:
\[
\delta_E(\varepsilon) = \inf\{1-\frac{\|x+y\|}{2}:\|x|\leq1,\|y|\leq1,\|x-y|\geq \varepsilon\}.
\]
A function $\rho_E:[0,+\infty]\rightarrow[0,+\infty]$, which is called the modulus of smoothness of $E$  as follows:
\[
\rho_E(t)=\sup\{\frac{1}{2}(\|x+y\|+\|x-y\|)-1:\|x\|\leq1,\|y\|\leq t\}.
\]
\begin{definition}
A Banach space $E$ is said to be
\begin{enumerate}[(I)]
	\item uniformly convex if for any $x,y\in E$ with $\|x\|=\|y\|=1$ and $\|x-y\|\geq\varepsilon$, there exists $\eta=\eta(\varepsilon)>0$ for all $\varepsilon\in(0,2]$ such that $\|\frac{x+y}{2}\|\leq1-\eta.$ This is equivalent to $\delta_E(\varepsilon)>0$, for all $\varepsilon\in(0,2].$
	\item  uniformly smooth if and only if $\lim_{t\rightarrow 0}\frac{\rho_E(t)}{t}=0.$	
\end{enumerate}
\end{definition}

A Banach space is called $p$-uniformly convex if there exists a constant $c>0$ such that $\delta_E(\varepsilon)>c\varepsilon^p$ for all $\varepsilon\in(0,2]$, where the constant $\frac{1}{c}$ is called the $p$-uniformly convexity constant. It is obvious that a $p$-uniformly convex Banach space is uniformly convex. A Banach space is said to be $q$-uniformly smooth if there exists a constant $C_q>0$ such that $\rho_E(t)\leq C_qt^q$ for all $t>0$, where $C_q$ is the $q$-uniformly smoothness constant. In addition, $E$ is a $p$-uniformly convex and uniformly smooth Banach space if and only if its dual $E^*$ is a $q$-uniformly smooth and uniformly convex Banach space.

Let $E$ be a Banach space with the dual $E^*$. The duality mapping $J_E^p:E\rightarrow 2^{E^*}$ is defined by
$
J_E^p(x)=\{x^*\in E^*:\langle x,x^*\rangle=\|x\|^p, \|x^*\|=\|x\|^{p-1}\},$ $\ p>1,$ $\forall x\in E.
$

\begin{definition}\label{def2.2}
 For a G\^{a}teaux differentiable convex function $f:E\rightarrow R$, the function
\begin{equation}\label{eq1}
 \Delta_f(x,y):=f(y)-f(x)-\langle f^\prime(x),y-x\rangle,\ \forall x,y \in E
\end{equation}
 is called the Bregman distance of $x$ to $y$ with respect to the function $f$.
\end{definition}

In addition, the duality mapping $J_E^p$ is the derivative of the function $f_p(x)=\frac{1}{p}\|x\|^p$. Then the Bregman distance with respect to $f_p$ can be written as
 	\begin{align*}
 	\Delta_p(x,y)&=\frac{1}{q}\|x\|^p-\langle J_E^px,y\rangle+\frac{1}{p}\|y\|^p\\
 	&=\frac{1}{p}(\|y\|^p-\|x\|^p)+\langle J_E^px,x-y\rangle\\
 	&=\frac{1}{q}(\|x\|^p-\|y\|^p)-\langle J_E^px-J_E^py,y\rangle.
 	\end{align*}

\begin{definition}\label{def2.3}
Let $C$ be a nonempty closed convex subset of a Banach space $E$. The mapping $T:C\rightarrow E$ is said to be
\begin{enumerate}[(I)]
	\item left Bregman quasi-nonexpansive mapping if $F(T)\neq\emptyset$ and
	\begin{equation}\label{eqw}
	\Delta_p(Tx,x^*)\leq\Delta_p(x,x^*),\ \forall x\in C,\ x^*\in F(T);	
	\end{equation}
	\item firmly nonexpansive-like mapping if
	\begin{equation}\label{eqe}
	\langle Tx-Ty, J_E^p(x-Tx)-J_E^p(y-Ty)\rangle\geq 0,\ \forall x,y\in C.
	\end{equation}	
\end{enumerate}
\end{definition}
Obviously, if $E$ is a Hilbert space, the firmly nonexpansive-like mapping reduce to the firmly nonexpansive mapping, i.e., $\langle Tx-Ty,x-y\rangle\geq \|Tx-Ty\|^2, \forall x,y\in C$.
\begin{example}
\begin{enumerate}[(1)]
	\item Let $E$ be a smooth, strictly convex and reflexive Banach space and $C$ be a nonempty closed convex subset of $E$. Then, the metric projection $P_C$ is a firmly nonexpansive-like mapping.
	\item Let $E$ be a real number space $R$ with Euclidean norm. A mapping $T:[-10, 10]\rightarrow [-10, 10]$ is defined by
$
Tx=\frac{1}{4}x,$ $\forall x\in [-10, 10].
$
	For any $x,y\in [-10, 10]$, we easily get the following result
	\begin{align*}
	\langle Tx-Ty, J_E^p(x-Tx)-J_E^p(y-Ty)\rangle&=\langle \frac{1}{4}x-\frac{1}{4}y, (x-\frac{1}{4}x)^3-(y-\frac{1}{4}Ty)^3 \rangle\\
	&=\frac{1}{4}\times(\frac{3}{4})^3(x-y)(x^3-y^3)\\
	&=\frac{1}{4}\times(\frac{3}{4})^3(x-y)^2(x^2+xy+y^2)\geq 0,
	\end{align*}
which implies that $T$ is a firmly nonexpansive-like mapping.
\end{enumerate}
\end{example}

Let $T:C\rightarrow E$ be a mapping. A point $u$ is said to be an asymptotic fixed point of $T$ if there exists a sequence $\{x_n\}$ in $C$ such that $x_n\rightharpoonup u$ and $x_n-Tx_n\rightarrow 0$. The set of all asymptotic fixed points of $T$ is denoted by $\widehat{F}(T)$.

\begin{lemma}\label{lemma2.5}
 Let $E$ be a smooth, strictly convex and reflexive Banach space, and $C$ be a nonempty closed convex subset of $E$. $T:C\rightarrow E$ is a firmly nonexpansive-like mapping. Then $F(T)$ is a closed convex subset of $E$ and $\widehat{F}(T)=F(T)$.
\end{lemma}
 \begin{proof}
In order to prove that $F(T)$ is a closed convex set, we assume that $F(T)$ is nonempty. Let $\{x_n\}$ be a sequence in $F(T)$ such that $x_n\rightarrow u$. From the definition of $T$, we have
$
\langle x_n-Tu, -J_E^p(u-Tu)\rangle\geq 0.
$
This inequality is equivalent also to
\[
\|u-Tu\|^p\leq \langle x_n-u, J_E^p(Tu-u)\rangle\leq \|x_n-u\|\|u-Tu\|^{p-1}.
\]
Then we obtain $\|u-Tu\|=0$ as $n\rightarrow \infty$. This implies $u=Tu$. Hence, $u\in F(T)$ and $F(T)$ is closed.

Next, we show that $F(T)$ is convex. For any $x,y\in F(T)$ and $t\in (0,1)$, putting $u=tx+(1-t)y$. From the definition of $T$, we get
$
\langle x-Tu, -J_E^p(u-Tu)\rangle\geq 0
$
and
$
\langle y-Tu, -J_E^p(u-Tu)\rangle\geq 0.
$
Combine the above two inequalities, we have
\begin{align*}
	\langle tx+(1-t)y-Tu, -J_E^p(u-Tu)\rangle\geq 0&\Leftrightarrow \langle u-Tu, -J_E^p(u-Tu)\rangle\geq 0\\
	&\Leftrightarrow \|u-Tu\|^p\leq 0.
\end{align*}
This means that $u=Tu$. So, $F(T)$ is closed and convex.

Last, we show that $\widehat{F}(T)=F(T)$. It is obvious that $F(T)\subset \widehat{F}(T)$. Then, we only show that $\widehat{F}(T)\subset F(T)$. For any $z\in \widehat{F}(T)$, there exists a sequence $\{x_n\}$ in $C$ such that $x_n\rightharpoonup z$ and $x_n-Tx_n\rightarrow 0$. From the definition of $T$, we have
\[
\langle Tx_n-Tz, J_E^p(x_n-Tx_n)-J_E^p(z-Tz)\rangle \geq 0.
\]
This is equivalent to
\begin{align*}
	\langle Tx_n-Tz, J_E^p(x_n-Tx_n)\rangle &\geq\langle Tx_n-Tz, J_E^p(z-Tz)\rangle\\
	&=\langle Tx_n-z+z-Tz, J_E^p(z-Tz)\rangle\\
	&=\langle Tx_n-z, J_E^p(z-Tz)\rangle+\|z-Tz\|^p.
\end{align*}
The inequality can be transformed the following inequality
\begin{align*}
	\|z-Tz\|^p&\leq \langle z-Tx_n, J_E^p(z-Tz)\rangle+\langle Tx_n-Tz, J_E^p(x_n-Tx_n)\rangle\\
	&=\langle z-x_n, J_E^p(z-Tz)\rangle+\langle x_n-Tx_n, J_E^p(z-Tz)\rangle\\
	 & \quad +\langle Tx_n-x_n, J_E^p(x_n-Tx_n)\rangle+\langle x_n-Tz, J_E^p(x_n-Tx_n)\rangle.
\end{align*}
From the setting of $n\rightarrow \infty$, we have $\|z-Tz\|=0$. Hence, $z=Tz$, i.e., $z\in F(T)$.	
\end{proof}
\begin{definition}
Let $C$ be a nonempty closed convex subset of a real Banach space $E$. A mapping $T:C\rightarrow C$ is closed (or $T$ has closed graph), that is, if the sequence $\{x_n\}$ in $C$ converges strongly to a point $x\in C$ and $Tx_n\rightarrow y$, then $Tx=y$.
\end{definition}
\begin{lemma}
\cite{s33} Let $E$ be a Banach space and $J_E^p$ be the duality mapping of $E$. Then, the following statements hold:
\begin{enumerate}[(I)]
	\item $J_E^p(x)$ is nonempty bounded closed and convex, for any $x\in E$;
	\item if $E$ is a reflexive Banach space, then $J_E^p$ is a mapping from $E$ onto $E^*$;
	\item if $E$ is a smooth Banach space, then $J_E^p$ is single valued;
	\item if $E$ is a uniformly smooth Banach space, then $J_E^p$ is norm-to-norm uniformly continuous on each bounded subset of $E$.
\end{enumerate}
\end{lemma}
\begin{remark}\label{remark2.8}
 By the definition of $\Delta_p$, we easily obtain
\begin{equation}\label{eqr}
\Delta_p(x,y)=\Delta_p(x,z)+\Delta_p(z,y)+\langle z-y,J_E^px-J_E^pz\rangle,\ \forall x,y,z\in E,
\end{equation}
and
\begin{equation}\label{eqt}
\Delta_p(x,y)+\Delta_p(y,x)=\langle x-y,J_E^px-J_E^py\rangle,\ \forall x,y,z\in E.
\end{equation}
In addition, it is easy to see from the above that the Bregman distance is not symmetrical, and for $p$-uniformly convex Banach spaces, we have
\begin{equation}\label{eqy}
\tau\|x-y\|^p\leq\Delta_p(x,y)\leq\langle x-y,J_E^px-J_E^py\rangle,\ \forall x,y\in E, \tau>0.
\end{equation}
This indicates that Bregman distance is non-negative.
\end{remark}

\begin{definition}
$\Pi_C:E\rightarrow C$ is said to be the Bregman projection mapping, that is,
\begin{equation}\label{equ}
\Pi_Cx=\underset{y\in C}{\operatorname{argmin}}\ \Delta_p(x,y),\ \forall x\in E.
\end{equation}
In other words, $\Pi_Cx$ corresponds a unique element $x_0\in C$ such that
\[
\Delta_p(x,x_0)=\underset{y\in C}{\min}\ \Delta_p(x,y),\ \forall x\in E.
\]
The Bregman projection can also be characterized by the following inequality
\begin{equation}\label{eqi}
\langle J_E^px-J_E^p\Pi_Cx,z-\Pi_Cx\rangle\leq0,\ \forall z\in C,
\end{equation}
this is equivalent to
\begin{equation}\label{eqo}
\Delta_p(\Pi_Cx,z)\leq\Delta_p(x,z)-\Delta_p(x,\Pi_Cx),\ \forall z\in C.
\end{equation}
\end{definition}

\begin{lemma}\label{lemmaq}
 \cite{s34} Let $E$ be a $q$-uniformly smooth Banach space with the $q$-uniformly smoothness constant $C_q>0$. For any $x,y\in E$, the following inequality holds:
\[
\|x-y\|^q\leq\|x\|^q-q\langle y,J_{E}^qx\rangle+C_q\|y\|^q.
\]
\end{lemma}
\begin{lemma}\label{lemmaw}
\cite{s26}Let $E$ be a $p$-uniformly convex and uniformly smooth Banach space and with its dual $E^*$, $J_E^p$ and $J_{E^*}^q$ are the duality mapping of $E$ and $E^*$, respectively. For any $\{x_n\}\subset E$, and $\{t_n\}\subset (0,1)$ with $\Sigma _{n=1}^Nt_n=1$, the following inequality holds.
$
\Delta_p(J_{E^*}^q(\sum _{n=1}^Nt_nJ_E^p(x_n)),x)\leq \sum _{n=1}^Nt_n\Delta_p(x_n,x),$ $\forall x\in E.
$
\end{lemma}
\begin{lemma}\label{lemmae}
 Let $E$ be a $p$-uniformly convex and uniformly smooth real Banach space, and $C_1=E$. Then, for any sequences $\{y_n\}$,$\{z_n\}$ and $\{w_n\}$ in $E$ the set
\[
C_{n+1}=\{u\in C_{n}:\Delta_p(y_n,u)\leq\Delta_p(z_n,u)\leq\Delta_p(w_n,u)\}
\]
is closed and convex for each $n\geq 1$.
\end{lemma}
\begin{proof}
First, since $C_1=E$, $C_1$ is closed and convex. Then we assume that $C_n$ is a closed and convex. For each $u\in C_n$, by the definition of the function $\Delta_p$, we have
\[
\Delta_p(y_n,u)\leq\Delta_p(z_n,u)\Leftrightarrow 2\langle J_E^pz_n-J_E^py_n,u \rangle\leq\frac{1}{q}(\|z_n\|^p-\|y_n\|^p),
\]
and
\[
\Delta_p(z_n,u)\leq\Delta_p(w_n,u)\Leftrightarrow 2\langle J_E^pw_n-J_E^pz_n,u \rangle\leq\frac{1}{q}(\|w_n\|^p-\|z_n\|^p).
\]
Hence, we know that $C_{n+1}$ is closed. In addition, we can easily prove that $C_{n+1}$ is a convex. The proof is completed.
\end{proof}

\section{Main Results}
In the section, we assume that the following conditions are satisfied:
\begin{enumerate}[(1)]
	\item $E_1$ and $E_2$ are two $p$-uniformly convex and uniformly smooth real Banach spaces;
	\item $A:E_1\rightarrow E_2$ is a bounded linear operator with adjoint operator $A^*$;
	\item $T:E_1\rightarrow E_1$ is a closed left Bregmen quasi-nonexpansive mapping;
	\item $S:E_2\rightarrow E_2$ is a firmly nonexpansive-like mapping.
\end{enumerate}

In addition, $J_{E_1}^p$ and $J_{E_2}^p$ are the duality mappings of $E_1$ and $E_2$, respectively, and $J_{E_1^*}^q$ is the duality mapping of $E_1^*$. It is worth noting that $E_1^*$ and $E_2^*$ are two $q$-uniformly smooth and uniformly convex Banach spaces, and $J_{E_1}^p=(J_{E_1^*}^q)^{-1}$, where $1<q\leq 2\leq p<\infty$ with $\frac{1}{p}+\frac{1}{q}=1$.

{\bf Algorithm}\quad For given initial values $x_0,\ x_1\in C_1=E_1$, the sequence $\{x_n\}$ generated by the following iterative algorithm:
\begin{equation}\label{eqa}
\left\{
\begin{aligned}
&{w_n=J_{E_1^*}^q(J_{E_1}^px_n+\theta_n J_{E_1}^p(x_n-x_{n-1})),\quad}\\
&{z_n=J_{E_1^*}^q(J_{E_1}^pw_n-\gamma_nA^*J_{E_2}^p(I-S)Aw_n),\quad}\\
&{y_n=J_{E_1^*}^q(\alpha_nJ_{E_1}^pz_n+(1-\alpha_n)J_{E_1}^pTz_n),\quad}\\
&{C_{n+1}=\{u\in C_{n}:\Delta_p(y_n,u)\leq\Delta_p(z_n,u)\leq\Delta_p(w_n,u)\},\quad}\\
&{x_{n+1}=\Pi_{C_{n+1}}x_0,\quad}
\end{aligned}
\right.
\end{equation}
where $\Pi_{C_{n+1}}$ is a Bregman projection of $E_1$ onto $C_{n+1}$, $\{\gamma_n\}$ is a sequence of real number in $ (0,(\frac{q}{C_q\|A\|^q})^{\frac{1}{q-1}})$, where $\frac{1}{c}$ is the $p$-uniformly convexity constant and $C_q$ is the $q$-uniformly smoothness constant, the sequences of real number  $\{\alpha_n\}\subset[a,b]\subset (0,1)$ and $\{\theta_n\}\subset [c,d]\subset (-\infty, +\infty)$.

\begin{lemma}\label{lemmaa}
 Let $E_1$, $E_2$, $T$, $S$, $A$, $A^*$ and $J_{E_1}^p$, $J_{E_2}^p$, $J_{E_1^*}^q$ be the same as above. If $\Gamma=\{x^*|x^*\in F(T);Ax^*\in F(S)\}$, then $\Gamma\subseteq C_n$ for any $n\geq1$.
\end{lemma}
\begin{proof}
 If $\Gamma=\emptyset$, it is obvious that $\Gamma\subseteq C_n$. Conversely, for any $x^*\in \Gamma$, we have $x^*\in F(T)$ and $Ax^*\in F(S)$. According to Lemma \ref{lemmaw} and the definition of left Bregman quasi-nonexpansive mapping $T$, we easily obtain
\begin{equation}\label{eqs}
\begin{aligned}
\qquad\qquad\qquad\qquad\Delta_p(y_n,x^*)&=\Delta_p(J_{E_1^*}^q(\alpha_nJ_{E_1}^pz_n+(1-\alpha_n)J_{E_1}^pTz_n),x^*)\\
&\leq \alpha_n\Delta_p(z_n,x^*)+(1-\alpha_n)\Delta_p(Tz_n,x^*)\\
&\leq \Delta_p(z_n,x^*).
\end{aligned}
\end{equation}

Since $E_1$ is a $p$-uniformly convex and uniformly smooth real Banach space, then $E_1^*$ is a $q$-uniformly smooth and uniformly convex Banach space and $J_{E_1}^p=(J_{E_1^*}^q)^{-1}$. From the property of firmly nonexpansive-like mapping $S$, we easily obtain
$
\langle J_{E_2}^p(I-S)Aw_n,Ax^*-SAw_n\rangle\leq 0 .
$
Further, we have
\begin{equation}\label{eqd}
\begin{aligned}
\langle J_{E_2}^p(I-S)Aw_n,Ax^*-Aw_n\rangle&=\langle J_{E_2}^p(I-S)Aw_n,Ax^*-SAw_n+SAw_n-Aw_n\rangle\\
&=-\|(I-S)Aw_n\|^p+\langle J_{E_2}^p(I-S)Aw_n,Ax^*-SAw_n\rangle\\
&\leq-\|(I-S)Aw_n\|^p.
\end{aligned}
\end{equation}
Again from \eqref{eqo}, \eqref{eqd} and Lemma \ref{lemmaq}, we obtain
\begin{align*}
\Delta_p(z_n,x^*)&=\Delta_p(J_{E_1^*}^q(J_{E_1}^pw_n-\gamma_nA^*J_{E_2}^p(I-S)Aw_n),x^*)\\
&=\frac{1}{q}\|J_{E_1^*}^q(J_{E_1}^pw_n-\gamma_nA^*J_{E_2}^p(I-S)Aw_n)\|^p+\frac{1}{p}\|x^*\|^p\\
&\quad -\langle J_{E_1}^pw_n-\gamma_nA^*J_{E_2}^p(I-S)Aw_n,x^*\rangle\\
&=\frac{1}{q}\|J_{E_1}^pw_n-\gamma_nA^*J_{E_2}^p(I-S)Aw_n\|^q+\frac{1}{p}\|x^*\|^p\\
&\quad -\langle J_{E_1}^pw_n,x^*\rangle+\gamma_n\langle J_{E_2}^p(I-S)Aw_n,Ax^*\rangle\\
&\leq\frac{1}{q}\|J_{E_1}^pw_n\|^q+\frac{1}{p}\|x^*\|^p-\langle J_{E_1}^pw_n,x^*\rangle-\gamma_n\langle J_{E_2}^p(I-S)Aw_n, Aw_n\rangle \\
&\quad +\gamma_n\langle J_{E_2}^p(I-S)Aw_n,Ax^*\rangle+\frac{C_q(\gamma_n\|A\|)^q}{q}\|J_{E_2}^p(I-S)Aw_n\|^q\\
&=\Delta_p(w_n,x^*)+\gamma_n\langle J_{E_2}^p(I-S)Aw_n,Ax^*-Aw_n\rangle\\
&\quad +\frac{C_q(\gamma_n\|A\|)^q}{q}\|(I-S)Aw_n\|^p.
\end{align*}
Since $\{\gamma_n\}$ is a real number sequence contained in $(0,(\frac{q}{C_q\|A\|^q})^{\frac{1}{q-1}})$, we get
\begin{equation}\label{eqf}
\Delta_p(z_n,x^*)\leq\Delta_p(w_n,x^*)+(\frac{C_q(\gamma_n\|A\|)^q}{q}-\gamma_n)\|(I-S)Aw_n\|^p\leq\Delta_p(w_n,x^*).
\end{equation}
From \eqref{eqs} and \eqref{eqf}, we have $x^*\in C_{n+1}$, that is, $\Gamma\subseteq C_n$, $\forall n\geq1$.
\end{proof}

\begin{theorem}\label{thm3.2}
Let $E_1$, $E_2$, $T$, $S$, $A$, $A^*$ and $J_{E_1}^p$, $J_{E_2}^p$, $J_{E_1^*}^q$ be the same as above. If $\Gamma=\{x^*|x^*\in F(T);Ax^*\in F(S)\}\neq\emptyset$, the sequence $\{x_n\}$ generated by Algorithm (3.1) converges strongly to a point $z=\Pi_\Gamma x_0\in \Gamma$.
\end{theorem}
\begin{proof}
By Lemma \ref{lemmae} and Lemma \ref{lemmaa}, we know that $\Pi_{C_{n+1}}x_0$ is well defined and $\Gamma\subseteq C_n$.
According to Algorithm \eqref{eqa}, we know $x_n=\Pi_{C_n}x_0$ and $x_{n+1}=\Pi_{C_{n+1}}x_0$ for each $n\geq1$. Using $\Gamma\subseteq C_n$ and \eqref{eqo}, we have
\begin{equation}\label{eqg}
\Delta_p(x_0,x_n)=\Delta_p(x_0,\Pi_{C_n}x_0)\leq\Delta_p(x_0,x^*),\ x^*\in \Gamma,\ \forall n\geq1.
\end{equation}
It implies that $\{\Delta_p(x_0,x_n)\}$ is bounded. Reusing \eqref{eqo}, we also have
\begin{equation}\label{eqh}
\begin{aligned}
\Delta_p(x_n,x_{n+1})=\Delta_p(\Pi_{C_n}x_0,x_{n+1})&\leq\Delta_p(x_0,x_{n+1})-\Delta_p(x_0,\Pi_{C_n}x_0)\\
&=\Delta_p(x_0,x_{n+1})-\Delta_p(x_0,x_n).
\end{aligned}
\end{equation}
It follows from that $\{\Delta_p(x_0,x_n)\}$ is nondecreasing. Hence, the limit $\underset{n\rightarrow \infty}{\lim}\Delta_p(x_0,x_n)$ exists, and
\begin{equation}\label{eqj}
\underset{n\rightarrow \infty}{\lim}\Delta_p(x_n,x_{n+1})=0.
\end{equation}
It follows from \eqref{eqy} that
\begin{equation}\label{eqk}
\underset{n\rightarrow \infty}{\lim}\|x_{n+1}-x_n\|=0.
\end{equation}
For some positive integers $m$, $n$ with $m\geq n$, we have $x_m=\Pi_{C_m}x_1\subseteq C_n$. Using \eqref{eqo}, we obtain
\begin{equation}\label{eql}
\begin{aligned}
\Delta_p(x_n,x_m)=\Delta_p(\Pi_{C_n}x_0,x_m)&\leq\Delta_p(x_0,x_m)-\Delta_p(x_0,\Pi_{C_n}x_0)\\
&=\Delta_p(x_0,x_m)-\Delta_p(x_0,x_n).
\end{aligned}
\end{equation}
Since the limit $\underset{n\rightarrow \infty}{\lim}\Delta_p(x_0,x_n)$ exists, it follows from \eqref{eql} that $\underset{n\rightarrow \infty}{\lim}\Delta_p(x_n,x_m)=0$ and $\underset{n\rightarrow \infty}{\lim}\|x_m-x_n\|=0$. Therefore, $\{x_n\}$ is a Cauchy sequence. Further, there exists a point $z\in C$ such that $x_n\rightarrow z$.

From Algorithm \eqref{eqa}, Definition \ref{def2.2} and Lemma \ref{lemmaq}, we have
\begin{equation}\label{eqaa}
\begin{aligned}
\Delta_p(w_n,x^*)
&=\frac{1}{q}\|J_{E_1^*}^q(J_{E_1}^px_n+\theta_nJ_{E_1}^p(x_n-x_{n-1}))\|^p+\frac{1}{p}\|x^*\|^p\\
&\quad -\langle J_{E_1}^px_n+\theta_nJ_{E_1}^p(x_n-x_{n-1}),x^*\rangle\\
&=\frac{1}{q}\|J_{E_1}^px_n+\theta_nJ_{E_1}^p(x_n-x_{n-1})\|^q+\frac{1}{p}\|x^*\|^p\\
&\quad -\langle J_{E_1}^px_n,x^*\rangle-\theta_n\langle J_{E_1}^p(x_n-x_{n-1}), x^*\rangle\\
&\leq\frac{1}{q}\|J_{E_1}^px_n\|^q+\frac{1}{p}\|x^*\|^p-\langle J_{E_1}^px_n,x^*\rangle-\theta_n\langle J_{E_1}^p(x_n-x_{n-1}), x^*\rangle \\
&\quad +\theta_n\langle J_{E_1}^p(x_n-x_{n-1}), x_n\rangle+\frac{C_q(\theta_n)^q}{q}\|J_{E_1}^p(x_n-x_{n-1})\|^q\\
&=\frac{1}{q}\|x_n\|^p+\frac{1}{p}\|x^*\|^p-\langle J_{E_1}^px_n,x^*\rangle-\theta_n\langle J_{E_1}^p(x_n-x_{n-1}), x^*\rangle \\
&\quad +\theta_n\langle J_{E_1}^p(x_n-x_{n-1}), x_n\rangle+\frac{C_q(\theta_n)^q}{q}\|x_n-x_{n-1}\|^p\\
&=\Delta_p(x_n,x^*)+\theta_n\langle J_{E_1}^p(x_n-x_{n-1}), x_n-x^*\rangle\\
&\quad +\frac{C_q(\theta_n)^q}{q}\|x_n-x_{n-1}\|^p.
\end{aligned}
\end{equation}
By virtue of Remark \ref{remark2.8} and the definition of $w_n$, we know
\begin{equation}\label{eqas}
\begin{aligned}
\Delta_p(w_n,x^*)&=\Delta_p(w_n,x_n)+\Delta_p(x_n,x^*)+\langle x_n-x^*, J_{E_1}^pw_n-J_{E_1}^px_n\rangle\\
&=\Delta_p(w_n,x_n)+\Delta_p(x_n,x^*)+\theta_n\langle x_n-x^*, J_{E_1}^p(x_n-x_{n-1})\rangle.
\end{aligned}
\end{equation}

By \eqref{eqaa} and \eqref{eqas}, we get
$
\Delta_p(w_n,x_n)\leq\frac{C_q(\theta_n)^q}{q}\|x_n-x_{n-1}\|^p.
$
Then, using \eqref{eqy}, \eqref{eqk} and the boundedness of the sequence $\{\theta_n\}$, we can obtain
\begin{equation}\label{eqad}
\underset{n\rightarrow \infty}{\lim}\|w_n-x_{n}\|=0.
\end{equation}
Using a similar method, we can get
\[
\Delta_p(w_n,x_{n+1})=\Delta_p(w_n,x_n)+\Delta_p(x_n,x_{n+1})+\langle x_n-x_{n+1}, J_{E_1}^pw_n-J_{E_1}^px_n\rangle.
\]
By setting $n\rightarrow \infty$, we have
\begin{equation}\label{eqaf}
\underset{n\rightarrow \infty}{\lim}\|w_n-x_{n+1}\|=0.
\end{equation}
Since $x_{n+1}=\Pi_{C_{n+1}}x_0\in C_{n+1}\subseteq C_n$, we have
\[
\Delta_p(y_n,x_{n+1})\leq\Delta_p(z_n,x_{n+1})\leq\Delta_p(w_n,x_{n+1}).
\]
According to \eqref{eqaf}, we obtain
\begin{equation}\label{eqag}
\underset{n\rightarrow \infty}{\lim}\Delta_p(y_n,x_{n+1})=0,\ \underset{n\rightarrow \infty}{\lim}\Delta_p(z_n,x_{n+1})=0,
\end{equation}
which implies that
$
\underset{n\rightarrow \infty}{\lim}\|y_n-x_{n+1}\|=0,\ \underset{n\rightarrow \infty}{\lim}\|z_n-x_{n+1}\|=0.
$
Hence
\begin{equation}\label{eqah}
\|x_n-z_n\|\leq\|x_{n+1}-x_n\|+\|x_{n+1}-z_n\|\rightarrow 0\ as\ n\rightarrow\infty,
\end{equation}
\begin{equation}\label{eqaj}
\|y_n-z_n\|\leq\|x_{n+1}-y_n\|+\|x_{n+1}-z_n\|\rightarrow 0\ as\ n\rightarrow\infty.
\end{equation}
In addition, since $E_1$ is a $p$-uniformly convex and uniformly smooth real Banach space, then $J_{E_1}^p$ is uniformly norm-to-norm continuous. It follows from Algorithm (3.1) and real number sequence $\{\alpha_n\}$ in $[a,b]\subset (0,1)$ that
\[
\underset{n\rightarrow \infty}{\lim}\|J_{E_1}^pTz_n-J_{E_1}^pz_n\|=\underset{n\rightarrow \infty}{\lim}\frac{1}{1-\alpha_n}\|J_{E_1}^py_n-J_{E_1}^pz_n\|=0,
\]
which also implies that $\underset{n\rightarrow \infty}{\lim}\|Tz_n-z_n\|=0$. By virtue of (3.12) and $x_n\rightarrow z$ we have $z_n\rightarrow z$. Using the closedness of $T$, we obtain $z_n\rightarrow z$ and $Tz=z$. From \eqref{eqy}, \eqref{eqf} and \eqref{eqah}, we have
\begin{align*}
(\gamma_n-\frac{C_q(\gamma_n\|A\|)^q}{q})\|(I-S)Aw_n\|^p&\leq\Delta_p(w_n,x^*)-\Delta_p(z_n,x^*)\\
&=\frac{1}{q}\|w_n\|^p-\frac{1}{q}\|z_n\|^p-\langle J_{E_1}^pw_n-J_{E_1}^pz_n,x^*\rangle\\
&=\Delta_p(w_n,z_n)+\langle J_{E_1}^pw_n-J_{E_1}^pz_n,z_n-x^*\rangle\\
&\leq(\|w_n-z_n\|+\|z_n-x^*\|)\|J_{E_1}^pw_n-J_{E_1}^pz_n\|.
\end{align*}
By setting of $n\rightarrow \infty$, the right-hand side of last inequality tends to 0. Moreover, $\{\gamma_n\}$ is a real number sequence contained in $(0,(\frac{q}{C_q\|A\|^q})^{\frac{1}{q-1}})$, we have
\[
\underset{n\rightarrow \infty}{\lim}\|(S-I)Aw_n\|=0.
\]
Since $A$ is a bounded linear operator, we have $Aw_n\rightarrow Az$. Again according to Definition \ref{def2.3} and Lemma \ref{lemma2.5}, we get $Az\in F(T)$. Then, from \eqref{eqi} and \eqref{eqa}, we have
\begin{equation}\label{eqak}
\langle J_{E_1}^px_0-J_{E_1}^px_n, p-x_n\rangle\leq 0,\ \forall p\in \Gamma.
\end{equation}
By setting $n\rightarrow \infty$ in \eqref{eqak}, we obtain
\begin{equation}\label{eqal}
\langle J_{E_1}^px_0-J_{E_1}^pz, p-z\rangle\leq 0,\ \forall p\in \Gamma.
\end{equation}
Again from \eqref{eqi}, we have $z=\Pi_\Gamma x_0$. Definitively, we obtain that $\{x_n\}$ generated by  {Algorithm} \eqref{eqa} strong converge $z=\Pi_\Gamma x_0\in \Gamma$. The proof is completed.
\end{proof}
\begin{remark}
The significant improvement of the results in this paper is using the shrinking projection algorithm with inertial effects to study the split common fixed problem in the framework of two $p$-uniformly convex and uniformly smooth Banach spaces, and the iterative sequence generated by Algorithm \eqref{eqa} is strongly converge to a solution of the split common fixed point problem.
\end{remark}
\begin{remark}
The field of study in this paper is the $p$-uniformly convex and uniformly smooth Banach space, which is more extensive than the Hilbert space \cite{s22,s23,s24}, the uniformly convex and $ 2 $-uniformly smooth Banach space \cite{s25} and the $ 2 $-uniformly convex and $ 2 $-uniformly smooth real Banach space \cite{s26}. The split common fixed point problem of firmly nonexpansive-like mappings in Theorem \ref{thm3.2} is more general than the split feasibility problem and fixed point problem in \cite{s26,s28}.
\end{remark}

As a corollary of Theorem \ref{thm3.2}, when $E_1$ and $E_2$ reduce to Hilbert spaces, the function $\Delta_p$ is equal to $\Delta_p(x,y)=\frac{1}{2}\|x-y\|^2$ and the Bregman projection $\Pi_C$ is equivalent to the metric projection $P_C$. Then, we obtain the following corollary.

\begin{corollary}
 Let $H_1$ and $H_2$ be two real Hilbert spaces. Let $T:H_1\rightarrow H_1$ be a closed quasi-nonexpansive mapping, $S:H_2\rightarrow H_2$ be a firmly nonexpansive mapping, and $A:H_1\rightarrow H_2$ be a bounded linear operator with adjoint operator $A^*$. For given initial values $x_0, x_1\in C_1=H_1$, the sequence $\{x_n\}$ generated by the following iterative algorithm:
\begin{equation}\label{eqaaa}
\left\{
\begin{aligned}
&{w_n=x_n+\theta_n(x_n-x_{n-1}),\quad}\\
&{z_n=w_n-\gamma_nA^*(I-S)Aw_n,\quad}\\
&{y_n=\alpha_nz_n+(1-\alpha_n)Tz_n,\quad}\\
&{C_{n+1}=\{u\in C_{n}:\|y_n-u\|\leq\|z_n-u\|\leq\|w_n-u\|\},\quad}\\
&{x_{n+1}=P_{C_{n+1}}x_0,\quad}
\end{aligned}
\right.
\end{equation}
where $P_{C_{n+1}}$ is a metric projection of $H_1$ onto $C_{n+1}$, the sequences of real numbers $\{\gamma_n\}\subset(0,\frac{2}{\|A\|^2})$, $\{\alpha_n\}\subset[a,b]\subset (0,1)$ and $\{\theta_n\}\subset[c,d]\subset(-\infty, +\infty)$.
If $\Gamma=\{x^*|x^*\in F(T), Ax^*\in F(S)\}\neq\emptyset$, the sequence $\{x_n\}$ generated by  \eqref{eqaaa} converges strongly to a point $z=P_\Gamma x_0\in \Gamma$.
\end{corollary}

\section{Applications}
\subsection{Split fixed point problems and variational inclusion problems}
Let $H$ be a real Hilbert space, and $B:H\rightarrow 2^H$ be a set-valued mapping with domain $D(B):=\{x\in H: B(x)\neq \emptyset\}$. An operator $B: H\rightarrow 2^H$ is called monotone if
$
\langle u-v, x-y\rangle\geq 0,$ $\forall u\in Bx,\ v\in By.
$
Further, $B$ is called maximal monotone if its graph $G(B)=\{(x,y): x\in D(B),\ y\in D(B)\}$ is not properly contained in the graph of any other monotone operator.

The problem of zero points of maximal monotone operator is
\[
\text{finding } x^*\in H, \text{ such that } 0\in B(x^*),
\]
where $B:H\rightarrow 2^H$ is a set-valued maximal monotone operator. Martinet \cite{s35} first proposed proximal point algorithm to solve the problem of a zero point of maximal monotone operator.

\begin{lemma}\label{lemma4.1.1}
 \cite{lab2} Let $H$ be a real Hilbert space. Let $B:H\rightarrow 2^H$ be a maximal monotone operator and $\mu>0$, and its associated resolvent of order $\mu$, defined by
$
J_\mu^B=(I+\mu A)^{-1},
$
where $I$ denotes the identity mapping. Then, the following properties are true:
\begin{enumerate}[(I)]
	\item For each $\mu>0$, $J_\mu^B$ is a sing-valued and firmly nonexpansive mapping;
	\item $D(J_\mu^B)=H$ and $F(J_\mu^B)=B^{-1}(0):=\{x\in D(B),\ 0\in Bx\}$.
\end{enumerate}
\end{lemma}
\begin{definition}
Let $H_1$ and $H_2$ be two Hilbert spaces, and $A:H_1\rightarrow H_2$ be a bounded linear operator. Let $B:H_1\rightarrow 2^{H_1}$ and $K:H_2\rightarrow 2^{H_2}$ be two set-valued mappings, $T:H_1\rightarrow H_1$ be a mapping. The split fixed point problem and variational inclusion problem is to find a point $x^*$ such that
\[
x^*\in F(T)\cap K^{-1}(0),\ Ax^*\in B^{-1}(0).
\]
For this problem, we propose the following theorem by the result in Theorem \ref{thm3.2}.
\end{definition}
\begin{theorem}
 Let $H_1$ and $H_2$ be two Hilbert spaces, and $A:H_1\rightarrow H_2$ be a bounded linear operator with adjoint operator $A^*$. Let $B:H_1\rightarrow 2^{H_1}$ and $K:H_2\rightarrow 2^{H_2}$ be two maximal monotone operators and $\mu>0$, $T:H_1\rightarrow H_1$ be a closed quasi-nonexpansive mapping. For given initial values $x_0,\ x_1 \in C_1=H_1$, the sequence $\{x_n\}$ generated by the following iterative algorithm:
\begin{equation}\label{eqz}
\left\{
\begin{aligned}
&{w_n=x_n+\theta_n(x_n-x_{n-1}),\quad}\\
&{z_n=w_n-\gamma_nA^*(I-J_\mu^B)Aw_n,\quad}\\
&{y_n=\alpha_nz_n+(1-\alpha_n)TJ_\mu^Kz_n,\quad}\\
&{C_{n+1}=\{u\in C_{n}:\|y_n-u\|\leq\|z_n-u\|\leq\|w_n-u\|\},\quad}\\
&{x_{n+1}=P_{C_{n+1}}x_0,\quad}
\end{aligned}
\right.
\end{equation}
where  $P_{C_{n+1}}$ is a metric projection of $H_1$ to $C_{n+1}$, the sequences of real numbers $\{\gamma_n\}\subset(0,\frac{2}{\|A\|^2})$, $\{\alpha_n\}\subset[a,b]\subset (0,1)$ and $\{\theta_n\}\subset[c,d]\subset(-\infty, +\infty)$.
If $\Gamma=\{x^*|x^*\in F(T)\cap K^{-1}(0), Ax^*\in B^{-1}(0)\}\neq\emptyset$, the sequence $\{x_n\}$ generated by iterative algorithm \eqref{eqz} converges strongly to a point $z=P_\Gamma x_0\in \Gamma$.
\end{theorem}

\begin{proof}
It is obvious that $TJ_\mu^K$ is closed quasi-nonexpansive mapping from the property of $T$ and Lemma \ref{lemma4.1.1}. Hence, the strong convergence theorem of  iterative algorithm \eqref{eqz} is obviously proved.
\end{proof}

\subsection{Split fixed point problems and equilibrium problems}
Let $C$ be a nonempty closed and convex subset of a Hilbert space $H$. Let bifunction $F:C\times C\rightarrow R$ satisfy the following conditions:
\begin{description}
\item[(A1)]$F(x,x)=0$, $\forall x\in C$;
\item[(A2)]$F(x,y)+F(y,x)\leq0$, $\forall x,y\in C$;
\item[(A3)]$\lim_{n\rightarrow\infty}F(tz+(1-t)x,y)\leq F(x,y)$, $\forall x,y,z\in C$;
\item[(A4)]For each $x\in C$, the function $y\mapsto F(x,y)$ is convex and lower semi-continuous.
\end{description}

Then, the so-called equilibrium problem for $F$ is to find a point $x^*\in C$ such that
$
F(x^*,x)\geq 0,$ $\forall x\in C,
$
and the set of solutions of equilibrium problem is denoted by $EP(F)$.
\begin{lemma}
 \cite{s37}\quad Let $C$ be a nonempty closed convex subset of a Hilbert space $H$, and let $F:C\times C\rightarrow R$ be a bifunction satisfying (A1)-(A4). Let $r>0$ and $x\in H$. Then there exists a point $z\in C$ such that
$
F(z,y)+\frac{1}{r}\langle y-z,z-x\rangle\geq 0,$ $\forall y\in C.
$
\end{lemma}

\begin{lemma}
	\cite{s37}\quad Assume that $F:C\times C\rightarrow R$ be a bifunction satisfying (A1)-(A4). For $r>0$ and $x\in H$, define a mapping $T_r^F:H\rightarrow H$ as follows:
$
T_r^F(x)=\{z\in C:F(z,y)+\frac{1}{r}\langle y-z,z-x\rangle\geq 0,\forall y\in C\},$ $\forall x\in H.
$
Then
\begin{enumerate}[(1)]
	\item $T_r^F$ is single-valued;
	\item $
	\|T_r^Fx-T_r^Fy\|^2\leq\langle T_r^Fx-T_r^Fy,x-y\rangle;
	$
	\item $F(T_r^F)=EP(F)$  is nonempty, closed and convex.
\end{enumerate}
\end{lemma}

\begin{definition}
Let $C$ and $Q$ be two nonempty closed convex subsets of real Hilbert spaces $H_1$ and $H_2$, respectively, and $A:H_1\rightarrow H_2$ be a bounded linear operator, $T:C\rightarrow C$ be a mapping, $F:C\times C\rightarrow R$ be a bifunction satisfying (A1)-(A4). The split fixed point problem and equilibrium problem is to find a point $x^*$ such that
\[
x^*\in F(T),\ Ax^*\in EP(F).
\]
\end{definition}
\begin{theorem}
 Let $C$ and $Q$ be two nonempty closed convex subsets of real Hilbert spaces $H_1$ and $H_2$, respectively, $T:C\rightarrow C$ be a closed quasi-nonexpansive mapping, $F:C\times C\rightarrow R$ be a bifunction satisfying (A1)-(A4), and $A:H_1\rightarrow H_2$ be a bounded linear operator with adjoint operator $A^*$. For given initial values $x_0,\ x_1\in C_1=C$, the sequence $\{x_n\}$ generated by the following iterative algorithm:
\begin{equation}\label{eq4.2}
\left\{
\begin{aligned}
&{w_n=x_n+\theta_n(x_n-x_{n-1}),}\\
&{z_n=P_C(w_n-\gamma_nA^*(I-T_r^F)Aw_n),}\\
&{y_n=\alpha_nz_n+(1-\alpha_n)Tz_n,}\\
&{C_{n+1}=\{u\in C_{n}:\|y_n-u\|\leq\|z_n-u\|\leq\|w_n-u\|\},}\\
&{x_{n+1}=P_{C_{n+1}}x_0,}
\end{aligned}
\right.	
\end{equation}
where $P_{C_{n+1}}$ is a metric projection of $H_1$ onto $C_{n+1}$, the sequences of real numbers $\{\gamma_n\}\subset(0,\frac{2}{\|A\|^2})$, $\{\alpha_n\}\subset[a,b]\subset (0,1)$ and $\{\theta_n\}\subset[c,d]\subset(-\infty, +\infty)$. If $\Gamma=\{x^*|x^*\in F(T),\ Ax^*\in EP(F)\}\neq\emptyset$, the sequence $\{x_n\}$ generated by iterative algorithm \eqref{eq4.2} converges strongly to a point $z=P_\Gamma x_0\in \Gamma$.
\end{theorem}

\section{Numerical examples}
In this section, we come up with some numerical examples to demonstrate the effectiveness and realization of convergence of Theorem \ref{thm3.2}. All codes were written in Matlab R2018b, and ran on a Lenovo ideapad 720S with 1.6 GHz Intel Core i5 processor and 8GB of RAM. Using numerical experiments, we will compare the convergence speed of our algorithm with the algorithm in \cite{s28}. Ma et al. \cite{s28} proved strong convergence theorems of split feasibility problems and fixed point problems of quasi-$\phi$-nonexpansive mapping in Banach spaces and introduced the following algorithm.
\begin{theorem}\label{thm}
 \cite{s28} Let $E_1$ be a $ 2 $-uniformly convex and $ 2 $-uniformly smooth real Banach space with best smoothness constant $k>0$, $E_2$ be a smooth, strictly convex, and reflective Banach space. Let $T:E_1\rightarrow E_1$ be a closed quasi-$\phi$-nonexpansive mapping with $F(T)\neq \emptyset$, $A:E_1\rightarrow E_1$ be a bounded linear operator with adjoint $A^*$, and $Q$ be a nonempty, closed, and convex subset of $E_2$. Let $x_1\in E_1$ and $C_1=E_1$, and $\{x_n\}$ be a sequence generated by
\[
\left\{
\begin{aligned}
&{z_n=J_1^{-1}(J_1x_n-\gamma A^*J_2(I-P_Q)Ax_n),}\\
&{y_n=J_1^{-1}(\alpha_nJ_1z_n+(1-\alpha_n)J_1Tz_n),}\\
&{C_{n+1}=\{u\in C_{n}:\phi(u,y_n)\leq\phi(u,x_n),\ \phi(u,z_n)\leq\phi(u,x_n)\},}\\
&{x_{n+1}=\Pi_{C_{n+1}}x_1,}
\end{aligned}
\right.
\]
where $P_Q$ is the metric projection of $E_2$ onto $Q$ and $\Pi_{C_{n+1}}$ is the generalized projection of $E_1$ onto $C_{n+1}$, $\{\alpha_n\}$ is a sequence in $(0,\delta]$, $\delta<1$, and $\gamma$ is a positive constant satisfying $0<\gamma<\frac{1}{\|A\|^2 k^2}$. If $\Gamma=\{x^*|x^*\in F(T), Ax^*\in Q\}\neq\emptyset$, then the sequence $\{x_n\}$ converges strongly to a point $z=\Pi_\Gamma x_1$.
\end{theorem}

To make sure the initial conditions of Theorem \ref{thm} and Theorem \ref{thm3.2} are consistent, the initial conditions are given as follows:

Let $E_1=R$ and $E_2=R\times R$ with the Euclidean norm,  and $C=[0,+\infty)$ and $Q=[0, +\infty)\times (-\infty,0]$. Let $A:E_1\rightarrow E_2$ be a bounded linear operator and defined as $Ax=(\frac{x}{2}, \frac{x}{3}),\ \forall x\in E_1$ with its adjoint $A^*(u,v)=\frac{u}{2}+\frac{v}{3},\ \forall (u,v)\in E_2$.

\begin{example}\label{ex5.2}
(Ma et al.) Let $T:C\rightarrow C$ be defined as $Tx=\frac{1}{4}x, \ \forall x\in C$, and $P_Q:E_2\rightarrow Q$ be a metric projection. In addition, we choose parameters $\gamma=1$, $\alpha_n=\frac{1}{7}$. For given initial value $x_1\in C_1=C$, the iterative algorithm in Theorem \ref{thm} can be simplified as
\[
\left\{
\begin{aligned}
&{Ax_n=(\frac{x_n}{2}, \frac{x_n}{3}),}\\
&{z_n=x_n-A^*(I-P_Q)Ax_n,}\\
&{y_n=\frac{1}{7}z_n+(1-\frac{1}{7})Tz_n,}\\
&{C_{n+1}=\{u\in C_{n}:|y_n-u|\leq|x_n-u|, |z_n-u|\leq|x_n-u|\},}\\
&{x_{n+1}=P_{C_{n+1}}x_1.}
\end{aligned}
\right.
\]
Then, the sequence $\{x_n\}$ converges strongly $0$.
\end{example}

\begin{example}\label{ex5.3}
 (Our algorithm with inertial effects in Theorem \ref{thm3.2}) Let $T:C\rightarrow C$ be defined as $Tx=\frac{1}{4}x, \ \forall x\in C$, and $S=P_Q:E_2\rightarrow Q$ be a metric projection. In addition, we choose parameters $\gamma_n=1$, $\alpha_n=\frac{1}{7}$ and $\theta_n=\frac{1}{2}$. For given initial values $x_0=x_1\in C_1=C$, the iterative algorithm in Theorem \ref{thm3.2} can be simplified as
\[
\left\{
\begin{aligned}
&{w_n=x_n+\frac{1}{2}(x_n-x_{n-1}),}\\
&{z_n=w_n-A^*(I-P_Q)Aw_n,}\\
&{y_n=\frac{1}{7}z_n+(1-\frac{1}{7})Tz_n,}\\
&{C_{n+1}=\{u\in C_{n}:|y_n-u|\leq|z_n-u|\leq|w_n-u|\},}\\
&{x_{n+1}=P_{C_{n+1}}x_0.}
\end{aligned}
\right.
\]
Then, the sequence $\{x_n\}$ converges strongly to $0$.

Next, taking initial values $x_1=6$, $x_1=3$ in Example \autoref{ex5.2} and $x_0=x_1=6$, $x_0=x_1=3$ in Example \autoref{ex5.3}. We get the following \autoref{table1} and \autoref{figure1} to demonstrate rate of convergence of \autoref{thm} and \autoref{thm3.2}.

\begin{table}[H]
\renewcommand\arraystretch{1}
\centering
\caption{Numerical results of \autoref{thm} and \autoref{thm3.2}}
\resizebox{\textwidth}{!}{ %

\begin{tabular}{ccccc} 
\hline
\toprule
\multirow{2}{*}{$ n $} &\multicolumn{2}{c}{Ma et al.}&\multicolumn{2}{c}{Algorithm 3.1}   \\
\cline{2-3}
\cline{4-5} 
   &  $x_1=6$ & $x_1=3$  & $x_0=x_1=6$  & $x_0=x_1=3$   \\
  \midrule
  0 & / & / & 6 & 3   \\
  1 &  6 & 3  & 6 & 3   \\
  2 & 3.952380952380953 & 1.976190476190476 & 3.619047619047619 & 1.809523809523809   \\
  3 & 2.603552532123961 & 1.301776266061980 & 1.895691609977324 & 0.947845804988662   \\
  4 & 1.715038572748323 & 0.857519286374162 & 0.935536119209588 & 0.467768059604794   \\
  5 & 1.129747631254848 & 0.564873815627424 & 0.448463346033803 & 0.224231673016901   \\
  6 & 0.744198836461527 & 0.372099418230763 & 0.211743715446802 & 0.105871857723401   \\
  \vdots & \vdots & \vdots  & \vdots & \vdots   \\
  21 & 0.001420045289491 & 0.000710022644746 & 0.000002127155711 & 0.000001063577855  \\
  22 & 0.000935426658951 & 0.000467713329475 & 0.000000986132411 & 0.000000493066206  \\
  23 & 0.000616193751531 & 0.000308096875766 & 0.000000457162771 & 0.000000228581385  \\
  24 & 0.000405905407755 & 0.000202952703877 & 0.000000211936762 & 0.000000105968381  \\
  25 & 0.000267382133680 & 0.000133691066840 & 0.000000098252052 & 0.000000049126026  \\

  \bottomrule
\end{tabular}
}
  \label{table1}
\end{table}

\begin{figure}[H]
\centering
\begin{minipage}[c]{0.5\textwidth}
\centering
\includegraphics[scale=0.4]{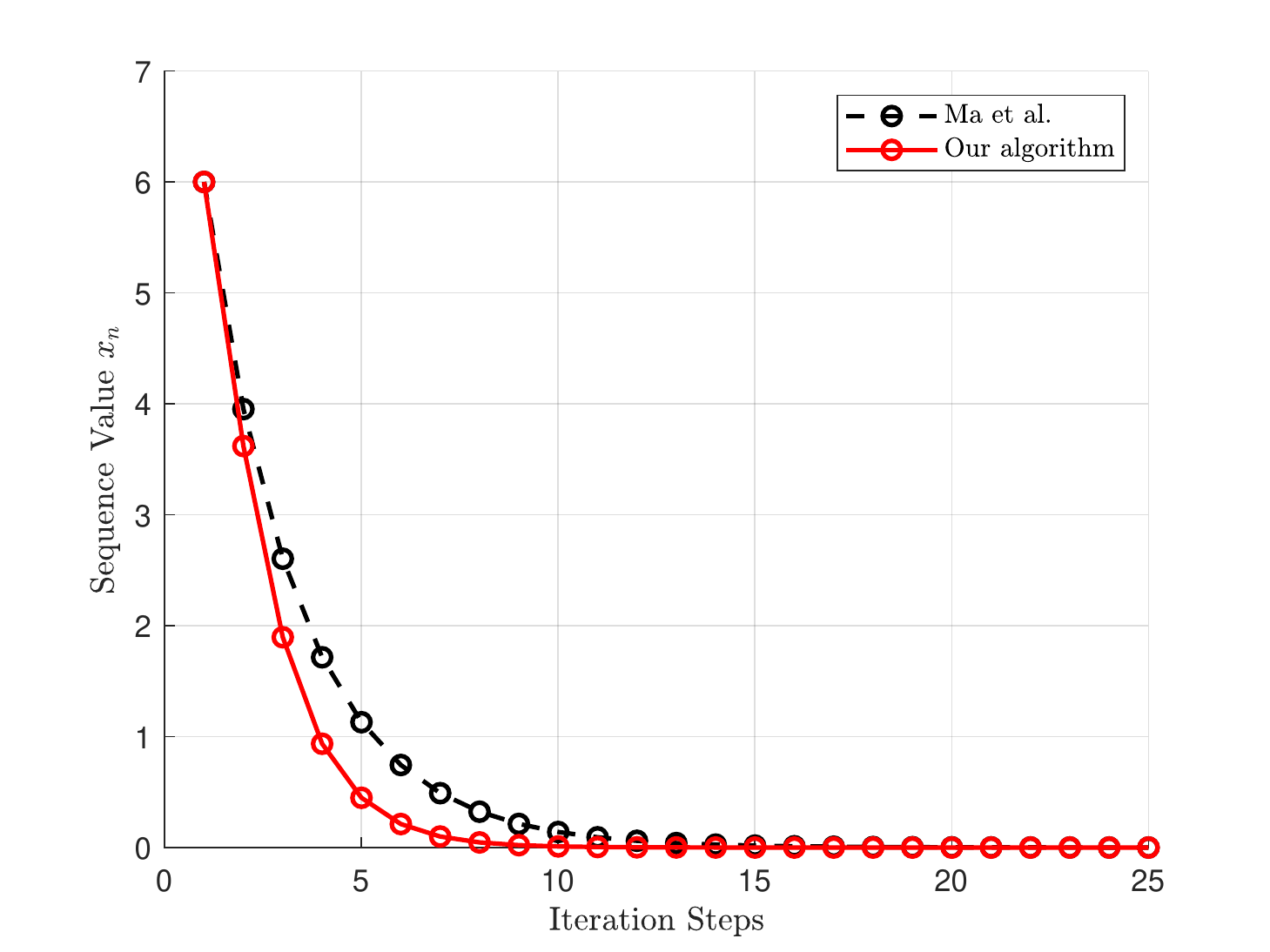}
\end{minipage}%
\begin{minipage}[c]{0.5\textwidth}
\centering
\includegraphics[scale=0.4]{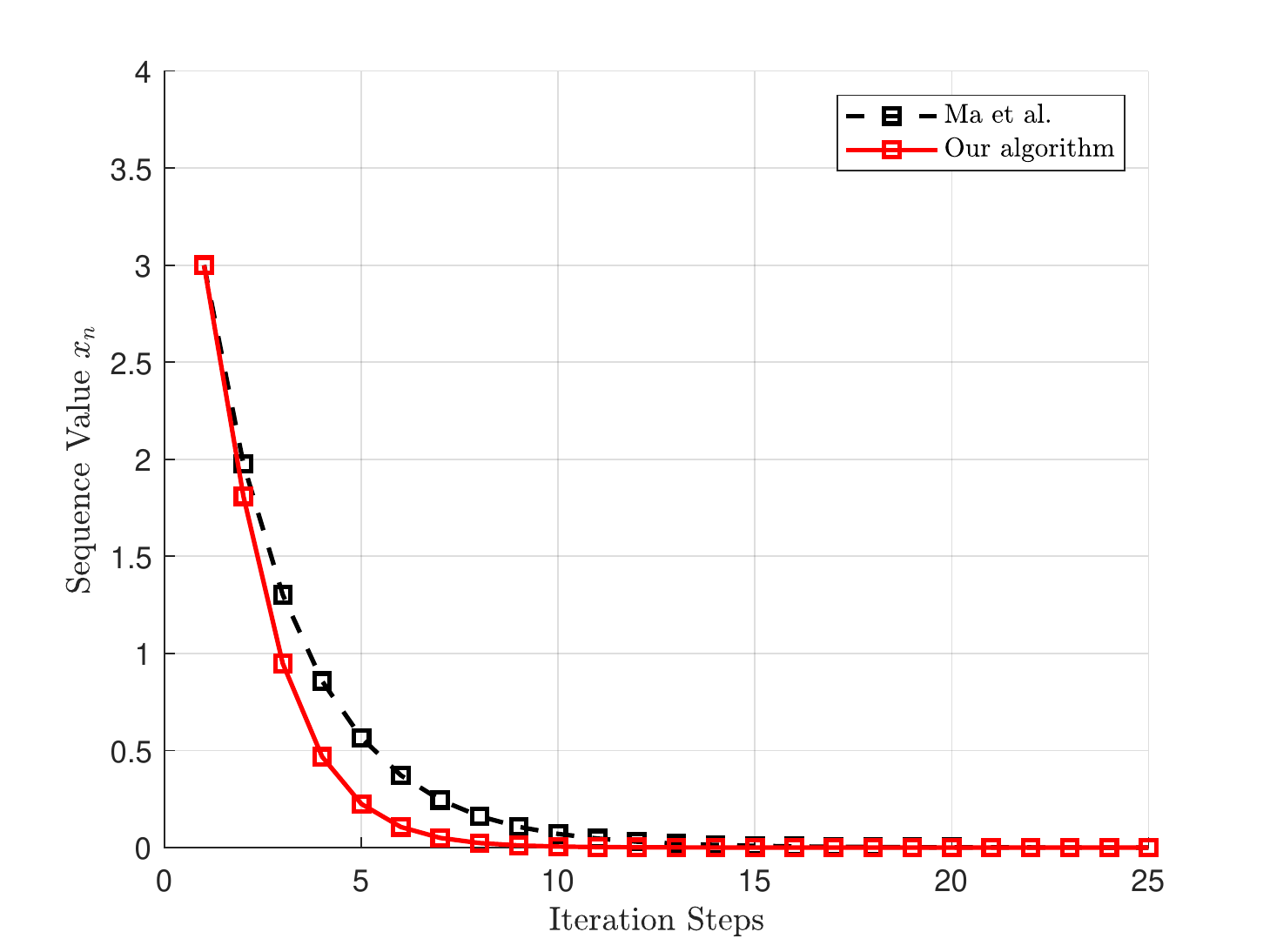}
\end{minipage}
\caption{Analysis of convergence speed of  \autoref{thm} and  \autoref{thm3.2}}
\label{figure1}
\end{figure}

It is worth noting that $\gamma_n$, $\alpha_n$ and $\theta_n$ are only constant step-size and initial point $x_0=x_1$(there is no using an inertial effect in the first iterative process) in Example \ref{ex5.3}.  To change this, we consider the following four cases of the step-size parameters $\gamma_n$, $\alpha_n$ and $\theta_n$, with the initial points $x_0=8$, $x_1=6$.
\begin{description}
	\item[Case 1:]  $\gamma_n=1$, $\alpha_n=\frac{1}{7}$ and $\theta_n=\frac{1}{2}$;
	\item[Case 2:]  $\gamma_n=\frac{2n+3}{2n}$, $\alpha_n=\frac{1}{7}$ and $\theta_n=\frac{1}{2}$;
	\item[Case 3:] $\gamma_n=\frac{2n+3}{2n}$, $\alpha_n=\frac{n}{7n+5}$ and $\theta_n=\frac{1}{2}$;
	\item[Case 4:]  $\gamma_n=\frac{2n+3}{2n}$, $\alpha_n=\frac{n}{7n+5}$ and $\theta_n=\frac{2n+1}{10n+2}$.
\end{description}

\begin{table}[H]
\renewcommand\arraystretch{1}

\centering
\caption{Numerical results of Case1-4}
\resizebox{\textwidth}{!}{ %
\begin{tabular}{ccccc} 
\hline
\toprule
{$ n $} &{Case 1}&{Case 2}&{Case 3}&{Case 4}   \\
  \midrule
  0 & 8 & 8 & 8 & 8   \\
  1 &  6 & 6  & 6 & 6   \\
  2 & 3.377777777777778 & 2.744444444444444 & 2.654166666666667 & 2.606770833333333   \\
  3 & 1.721058201058201 & 1.144272486772487 & 1.062511878654971 & 0.982533042700857   \\
  4 & 0.838240362811792 & 0.466087018140590 & 0.414427676387483 & 0.349299329383015   \\
  5 & 0.399106638351990 & 0.189975923807574 & 0.161779387393689 & 0.121333554378606   \\
  6 & 0.187756126213986 & 0.078232012035175 & 0.063949280327919 & 0.042025571691499   \\

  \vdots & \vdots & \vdots  & \vdots & \vdots   \\

  21 & 0.000001877748043 & 0.000000362109760 & 0.000000229101124 & 0.000000057064087  \\
  22 & 0.000000870508717 & 0.000000164098265 & 0.000000103109472 & 0.000000025326458  \\
  23 & 0.000000403560514 & 0.000000074452809 & 0.000000046477831 & 0.000000011270697  \\
  24 & 0.000000187087193 & 0.000000033815734 & 0.000000020979640 & 0.000000005027255  \\
  25 & 0.000000086732002 & 0.000000015373490 & 0.000000009481898 & 0.000000002246927  \\

  \bottomrule
\end{tabular}
}
  \label{table2}
\end{table}

\begin{figure}[H]
\centering
\begin{minipage}[c]{0.5\textwidth}
\centering
\includegraphics[scale=0.4]{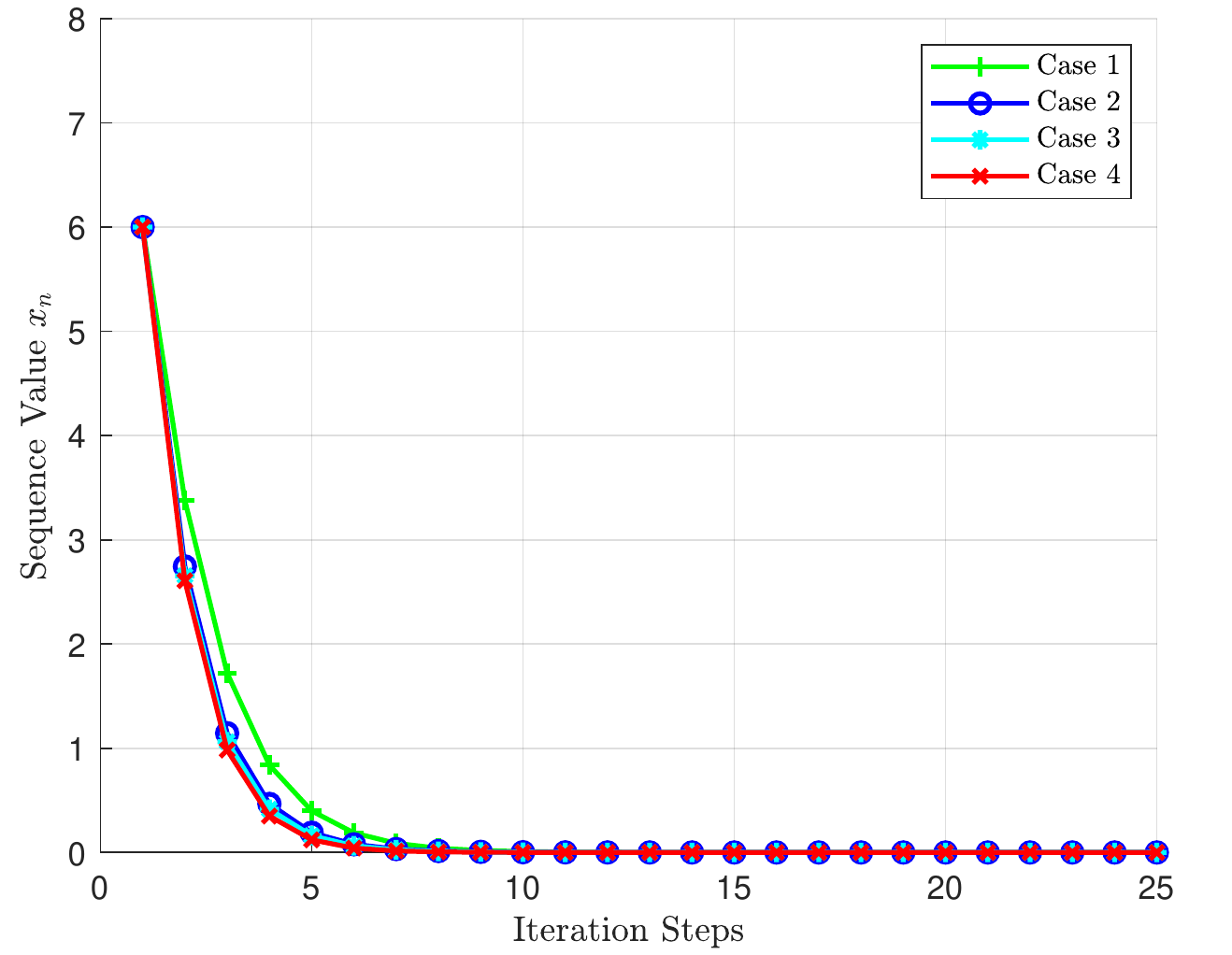}
\end{minipage}%
\begin{minipage}[c]{0.5\textwidth}
\centering
\includegraphics[scale=0.4]{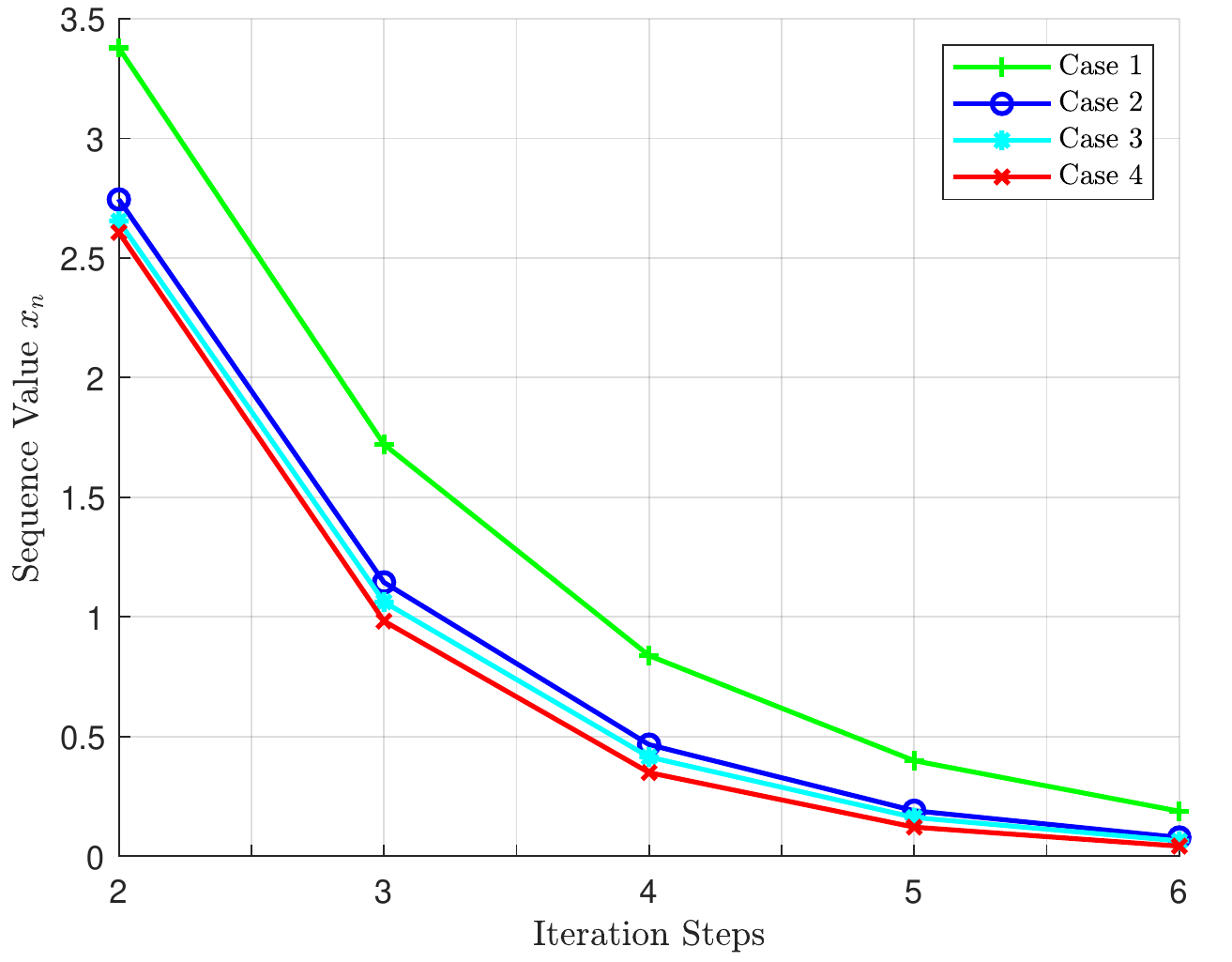}
\end{minipage}
\caption{Analysis of convergence speed of Case1-4}
\label{figure2}
\end{figure}
\end{example}

\begin{remark}
	\begin{enumerate}[(I)]
		\item From  \autoref{figure1} and \autoref{table1}, we found that the convergence rate of Theorem \ref{thm3.2} is faster than \autoref{thm}. This also shows the efficiency of our proposed Algorithm \eqref{eqa} in this paper.
		\item From \autoref{figure2} and \autoref{table2}, we know that the convergence speed of Theorem \ref{thm3.2} is improved under suitable condition of step-size.
	\end{enumerate}
\end{remark}

\end{document}